\documentclass{amsart}

\usepackage{pinlabel} 
\usepackage{amsmath}
\usepackage{amsfonts}
\usepackage{amssymb}
\usepackage{mathtools}
\usepackage{amscd}
\usepackage[all,cmtip]{xy}
\usepackage{amsthm}
\usepackage{comment}
\usepackage{anysize}
\usepackage{epsfig}
\usepackage{hyperref}
\usepackage[colorinlistoftodos]{todonotes}
\usepackage[a4paper, total={6in, 8in}]{geometry}
\usepackage[utf8]{inputenc}
\usepackage{xcolor}

\makeatletter
\providecommand\@dotsep{5}
\def\listtodoname{List of Todos}
\def\listoftodos{\@starttoc{tdo}\listtodoname}
\makeatother

\newtheorem{theorem}{Theorem}[section]
\newtheorem{proposition}[theorem]{Proposition}
\newtheorem{corollary}[theorem]{Corollary}
\newtheorem{lemma}[theorem]{Lemma}

  \theoremstyle{definition}
\newtheorem{definition}[theorem]{Definition}

\newtheorem{remark}[theorem]{Remark}

\newcommand{\dbE}{\mathbb{E}}

\newcommand{\dbH}{{\mathbb H}}

\newcommand{\dbR}{\mathbb{R}}

\newcommand{\dbZ}{\mathbb{Z}}

\newcommand{\calA}{{\mathcal A}}

\newcommand{\calF}{{\mathcal F}}

\newcommand{\calH}{{\mathcal H}}

\newcommand{\calO}{{\mathcal O}}

\newcommand{\vcyc}{V\text{\tiny{\textit{CYC}}}}

\newcommand{\fin}{F\text{\tiny{\textit{IN}}}}

\newcommand{\nbeq}{\begin{equation}}
\newcommand{\neeq}{\end{equation}}
\newcommand{\beq}{\begin{equation*}}
\newcommand{\eeq}{\end{equation*}}

\newcommand{\gdvcgamma}{\gdvc(\Gamma)}

\newcommand{\mbn}{\mathbb{N}}
\newcommand{\mbz}{\mathbb{Z}}

\newcommand{\mbr}{\mathbb{R}}
\newcommand{\mbc}{\mathbb{C}}
\newcommand{\mbe}{\mathbb{E}}

\newcommand{\mbs}{\mathbb{S}}
\newcommand{\cat}[1]{\mathrm{CAT(}#1\mathrm{)}}
\newcommand{\psl}{\mathrm{PSL}_2(\mbr)}

\newcommand{\pslt}{\widetilde{\mathrm{PSL}}_2(\mbr)}
\newcommand{\hyp}{\mathbb{H}^2}
\newcommand{\hypp}{\mathbb{H}^3}

\newcommand{\ihypp}{\mathrm{Isom}(\mathbb{H}^3)}
\newcommand{\nil}{\mathrm{Nil}}
\newcommand{\sol}{\mathrm{Sol}}

\newcommand{\hxr}{\mathbb{H}^2\times\mathbb{E}}

\newcommand{\fmly}{\mathcal{F}}

\newcommand{\mc}[1]{\mathcal{#1}}
\newcommand{\efin}{\underline{E}}
\newcommand{\evc}{\underline{\underline{E}}}

\newcommand{\hvc}{\underline{\underline{\mathrm{H}}}}
\newcommand{\hfin}{\underline{\mathrm{H}}}

\newcommand{\isom}{\mathrm{Isom}}

\DeclareMathOperator{\vrt}{vert}

\DeclareMathOperator{\cd}{cd}
\DeclareMathOperator{\vcd}{vcd}
\DeclareMathOperator{\cdfin}{\underline{cd}}
\DeclareMathOperator{\cdvc}{\underline{\underline{cd}}}

\DeclareMathOperator{\gd}{gd}
\DeclareMathOperator{\gdfin}{\underline{gd}}
\DeclareMathOperator{\gdvc}{\underline{\underline{gd}}}

\DeclareMathOperator{\vertex}{vert}
\DeclareMathOperator{\edge}{edge}

\begin{document}

\title{Virtually cyclic dimension for $3$-manifold groups}

\author{Kyle Joecken}
\address{Department of Mathematics, The Ohio State University, 100 Math Tower,
231 West 18th Avenue,
Columbus, OH 43210-1174, USA}
\email{joecken@gmail.com}

\author{Jean-Fran\c{c}ois Lafont}
\address{Department of Mathematics, The Ohio State University, 100 Math Tower,
231 West 18th Avenue,
Columbus, OH 43210-1174, USA}
\email{jlafont@math.ohio-state.edu}

\author{Luis Jorge S\'anchez Salda\~na}
\address{Department of Mathematics, The Ohio State University, 100 Math Tower,
231 West 18th Avenue,
Columbus, OH 43210-1174, USA}
\email{sanchezsaldana.1@osu.edu}


\date{}



\begin{abstract}
Let  $\Gamma$ be the fundamental group of a connected, closed, orientable $3$-manifold. We explicitly compute its virtually cyclic geometric dimension $\gdvcgamma$. Among the tools we use are the prime and JSJ decompositions of $M$, several push-out type constructions, as well as some Bredon cohomology computations. 
\end{abstract}

\maketitle

\section{Introduction}

Given a group $\Gamma$, we say that a collection of subgroups $\calF$ is called a \emph{family} if it is closed under conjugation and under taking subgroups. We say that a $\Gamma$-CW-complex $X$ is a model for the classifying space $E_\calF \Gamma$ if every isotropy group of $X$ belongs to $\calF$, and $X^H$ is contractible whenever $H$ belongs to $\calF$. Such a model always exists and it is unique up to $\Gamma$-homotopy equivalence. The geometric dimension of $\Gamma$ with respect to the family $\calF$, denoted $\gd_\calF (\Gamma)$, is the minimum dimension $n$ such that $\Gamma$ admits an $n$-dimensional model for $E_\calF \Gamma$.

Classical examples of families are the family that consists only of the trivial subgroup $\{1\}$, and the family $\fin$ of finite subgroups of $\Gamma$. A group is said to be \emph{virtually cyclic} if it contains a cyclic subgroup (finite or infinite) of finite index. We will also consider the family $\vcyc$ of virtually cyclic subgroups of $\Gamma$. These three families are relevant to the Farrell--Jones and the Baum--Connes isomorphism conjectures.

In the present paper we study $\gd_{\vcyc}(\Gamma)$ (also denoted $\gdvcgamma$) when $\Gamma$ is the fundamental group of an orientable, closed, connected $3$-manifold. We call a group \emph{non-elementary} if it is not virtually cyclic. Our main result is the following theorem.

\begin{theorem}\label{main:theorem}
Let $M$ be a connected, closed, oriented $3$-manifold, and let $\Gamma=\pi_1(M)$ be the fundamental group of $M$. Then $\gdvc (\Gamma) \leq 4$.
Moreover, we can classify $\gdvc(\Gamma)$ as follows:

\begin{enumerate}
\item $\gdvcgamma=0$ if and only if $\Gamma$ is virtually cyclic;

\item $\gdvcgamma=2$ if and only if $\Gamma$ is a non-elementary free product of virtually cyclic groups;

\item $\gdvc (\Gamma)=4$ if and only if $\Gamma$ contains a $\mathbb Z^3$ subgroup;

\item $\gdvc (\Gamma)=3$ in all other cases.

\end{enumerate}
\end{theorem}

Since we are dealing with $3$-manifold groups, the purely group theoretic description given above also corresponds to the following more geometric characterization 
of the virtually cyclic geometric dimension.

\begin{corollary}\label{main:corollary}
Let $M$ be a connected, closed, oriented $3$-manifold, and let $M=P_1 \# \cdots \# P_k$ be the prime decomposition of $M$. Let $\Gamma=\pi_1(M)$ be the fundamental group of $M$. Then we can classify $\gdvc(\Gamma)$ as follows:

\begin{enumerate}
\item $\gdvcgamma=0$ if and only if $M$ is modeled on $S^3$ or $S^2\times \dbR$;

\item $\gdvcgamma=2$ if and only if  every $P_i$ in the prime decomposition of $M$ is modeled on $S^3$ or $S^2\times \dbE$, and either: (1) $k>2$, or (2) $M=P_1\# P_2$ with $|\pi_1(P_1)|>2$;

\item $\gdvc (\Gamma)=4$ if and only if at least one of the prime components $P_i$ is modeled on $\dbE^3$;

\item $\gdvc (\Gamma)=3$ in all other cases.

\end{enumerate}
\end{corollary}

Our main tools for proving Theorem \ref{main:theorem} is the Kneser--Milnor prime decomposition 
and the Jaco--Shalen--Johannson JSJ decomposition of a $3$-manifold, the push-out constructions 
associated to acylindrical splittings from \cite{LO09_2}, the theory of Seifert fibered and hyperbolic manifolds, and some Bredon cohomology
computations.

\vskip 10pt

The paper is organized as follows. In Section \ref{sec-background} we review the notions of geometric 
and cohomological dimensions for families of subgroups and some of the basic notation of Bass--Serre 
theory. Section \ref{sec-3mflds} is devoted to recalling some basics of $3$-manifold theory, such as
the prime decomposition of Kneser--Milnor and the JSJ decomposition of Jaco--Shalen--Johannson. 
In Section \ref{sec-push-out} we state some useful push-out constructions that will help us 
construct new classifying spaces out of old ones. We also relate the geometric dimension of the 
fundamental group of a graph of groups with that of the vertex and edges groups, provided that the 
splitting is acylindrical. This allows us to establish Theorem 
\ref{reduction:irreducible}, which reduces the calculation to the case of prime 
manifolds. We then 
analyze case-by-case the situation when $M$ is a Seifert fibered space (Section \ref{sec-Seifert-fibered}) 
or a hyperbolic manifold (Section \ref{sec-hyperbolic}), with the results of these analyses summarized 
in Tables \ref{tabla2} and \ref{tabla}. In Section \ref{sec-twisted-bundles}, we focus on $3$-manifolds 
whose JSJ decomposition only contains pieces that are Seifert fibered with Euclidean orbifold base -- 
and show that these manifolds are always geometric. 
The main result of Section \ref{sec-reduction} is Theorem \ref{reduction:seifert:hyp}, where for a non-geometric
prime $3$-manifold, we show that the JSJ decomposition gives rise to an acylindrical splitting.  
Section \ref{sec:Bredon} then finishes the computation of the virtually cyclic geometric dimension for all
the prime manifolds that are not geometric. 
Finally, in Section \ref{sec-main-thm}, we bring these results together and prove Theorem \ref{main:theorem}.

\vskip 10pt

\centerline{\bf Acknowledgments}

\vskip 10pt

J.-F.L. was partially supported by the NSF, under grant DMS-1812028.  L.J.S.S. was supported by the 
National Council of Science and Technology (CONACyT) of Mexico via the program ``Apoyo para 
Estancias Posdoctorales en el Extranjero Vinculadas a la Consolidaci\'on de Grupos de Investigaci\'on 
y Fortalecimiento del Posgrado Nacional". The third author also wants to thank the OSU math department 
for its hospitality, and specially the second named author.

\vskip 20pt


\section{Preliminaries}\label{sec-background}

\subsection{Virtually cyclic geometric and cohomological dimension}
Let $\Gamma$ be a discrete group.  A nonempty set $\fmly$ of subgroups of
$\Gamma$ is called a \emph{family} if it
is closed under conjugation and passing to subgroups. We
call a $\Gamma$-CW-complex a \emph{model} for $E_{\fmly}\Gamma$ if for every
$H\leq \Gamma$:
\begin{enumerate}
  \item $H\notin \fmly\Rightarrow X^{H}=\emptyset$ ($X^{H}$ is the
  $H$-fixed subcomplex of $X$);
  \item $H\in\fmly\Rightarrow X^H$ is contractible.
\end{enumerate}
Such models always exist for every discrete group $\Gamma$ and every family of subgroups $\calF$. Moreover, every pair of models for $E_\calF \Gamma$ are $\Gamma$-homotopically equivalent. The geometric dimension $\gd_\calF(\Gamma)$ of $\Gamma$ with respect to the family $\calF$, is the minimum $n$ for which an $n$-dimensional model for $E_\calF \Gamma$ exists.

On the other hand, given $\Gamma$ and $\calF$ we have the so-called restricted orbit category $\calO_\calF \Gamma$, which has as objects the homogeneous $\Gamma$-spaces $\Gamma/H$, $H\in \calF$, and morphisms consisting of  $\Gamma$-maps between them. We define an $\calO_\calF \Gamma$-module to be a functor from $\calO_\calF \Gamma$ to the category of abelian groups, while a morphism between two $\calO_\calF \Gamma$-modules is a natural transformation of the underlying functors. Denote by $\calO_\calF \Gamma$-mod the category of $\calO_\calF \Gamma$-modules, which is an abelian category with enough projectives. Thus we can define a $\Gamma$-cohomology theory for $\Gamma$-spaces $H_\calF^*(-;F)$ for every $\calO_\calF\Gamma$-module $F$.  The
\emph{Bredon cohomological dimension of $G$}---denoted $\cd_{\fmly}(G)$---is the
largest nonnegative $n\in\mbz$ for which the Bredon cohomology group $H^n_{\fmly}(G;F)=H^n_{\fmly}(E_\calF G;F)$
is nontrivial for some $M\in\textrm{Mod-}\mathcal{O}_{\fmly}G$.

In the present work we are mainly concerned with the family $\vcyc$ of virtually cyclic subgroups. A highly related family is the family $\fin$ of finite subgroups. We will denote $E_{\vcyc} \Gamma$ (resp. $\gd_{\vcyc}(\Gamma)$, $\cd_{\vcyc}(\Gamma)$, $H_{\vcyc}^*$)  and $E_{\fin} \Gamma$ (resp. $\gd_{\fin}(\Gamma)$, $\cd_{\fin}(\Gamma)$, $H_{\fin}^*$) as $\evc \Gamma$ (resp. $\gdvc(\Gamma)$, $\cdvc(\Gamma)$, $\underline{\underline{H}}^*$) and $\efin \Gamma$ (resp. $\gdfin(\Gamma)$, $\cdfin(\Gamma)$, $\underline{H}^*$)  respectively. We also call $\gdvcgamma$ \emph{the virtually cyclic (or VC) geometric dimension of $\Gamma$}.

\begin{lemma}
We have the following properties of the geometric dimension:

\begin{enumerate}\label{properties:gd}
    \item If $H\leq G$, then $\gd_{\calF\cap H}(H)\leq \gd_\calF(G)$ for every family $\calF$ of $G$.
    \item For every group $G$ and every family of subgroups $\calF$ we have
    \[
    \cd_\calF(G)\leq \gd_\calF(G)\leq \max\{3,\cd_\calF(G)\}.
    \]
    In particular, if $\cd_\calF(G)\geq 3$, then $\cd_\calF(G)=\gd_\calF(G)$.
    \item If $G$ is an $n$-crystallographic group then $\gdvc G=n+1$
\end{enumerate}
\end{lemma}

\begin{proof} 
Statement (1) Follows from the observation that a model for $E_\calF G$ is also 
a model for $E_{H\cap \calF}H$, just by restricting the $G$-action to the subgroup 
$H\leq G$. Statement (2) is the main result of \cite{LM00}, while statement (3) 
follows from \cite{CFH06}.
 \end{proof}

\subsection{Graphs of groups} In this subsection we give a quick review of Bass-Serre theory, referring the reader to \cite{Se03} for more details. A graph (in the sense of Bass and Serre) consists of a set of vertices $V=\vertex Y$, a set of (oriented) edges $E=\edge Y$, and two maps $E\to V\times V$, $y\mapsto (o(y),t(y))$, and $E\to E$, $y\mapsto \overline{y}$ satisfying $\overline{\overline{y}}=y$, $\overline{y}\neq y$, and $o(y)=t(\overline{y})$. The vertex $o(y)$ is called the \emph{origin} of $y$, and the vertex $t(y)$ is called the \emph{terminus} of $y$.

An orientation of a graph $Y$ is a subset $E_+$ of $E$ such that $E=E_+\bigsqcup \overline{E}_+$. We can define  \emph{path} and \emph{circuit} in the obvious way.

A \emph{graph of groups} $\mathbf{Y}$ consists of a graph $Y$, a group $Y_P$ for each $P\in \vertex Y$, and a group $Y_y$ for each $y\in \edge Y$, together with monomorphisms $Y_y\to Y_{t(y)}$. One requires in addition $Y_{\overline{y}}=Y_y$.

Suppose that the group $G$ acts without inversions on a graph $X$, i.e. for every $g\in G$ and $x\in \edge X$ we have $g x\neq\overline{x} $. Then we have an induced graph of groups with underlying graph $X/G$ by  associating to each vertex (resp. edge) the isotropy group of a preimage under the quotient map $X\to X/G$.

Given a graph of groups $\mathbf{Y}$, one of the classic theorems of Bass-Serre theory provides the existence of a group $G=\pi_1(\mathbf{Y})$, called the \emph{fundamental group of the graph of groups $\mathbf{Y}$} and a tree $T$ (a graph with no cycles), called the \emph{Bass-Serre tree of $\mathbf{Y}$}, such that $G$ acts on $T$ and the induced graph of groups is isomorphic to $\mathbf{Y}$. The identification $G=\pi_1(\mathbf{Y})$ is called a \emph{splitting} of $G$.

Analogously we can define a graph of spaces $\mathbf{X}$ as a graph $X$, CW-complexes $X_P$ and $X_y$ for each vertex $P$ and each edge $y$, and closed cellular embeddings $X_y\to X_P$ if either $P=t(y)$ or $P=o(y)$. We also assume that the images of the embeddings are disjoint.  In this case we will have a CW-complex, called the \emph{geometric realization}, that is assembled by gluing the ends of the product space $X_y\times [0,1]$ to the spaces $X_{t(y)}$, $X_{o(y)}$.

Finally, given a graph of spaces $\mathbf{X}$ with $\pi_1(X_y)\to \pi_1(X_P)$ injective, there is an associated graph of groups $\mathbf{Y}$ with the same underlying graph and whose vertex (resp. edge) groups are the fundamental groups of the corresponding vertex (resp. edges) CW-complexes. Then, as a generalization of the Seifert--van Kampen theorem, we have that the fundamental group of the geometric realization of $\mathbf{X} $ is naturally isomorphic to the fundamental group of the graph of groups $\mathbf{Y}$.

\section{3-manifolds and decompositions}\label{sec-3mflds}

In this section we will review some $3$-manifold theory. For more details see \cite{Sc83}, \cite{Mo05}.

\subsection{Seifert fibered spaces}
A \emph{trivial fibered solid torus} is the usual product $S^1\times D^2$ with
the product foliation by circles $S^1\times\{y\}, y\in D^2$.  A \emph{fibered
solid torus} is a solid torus with a foliation by circles which is finitely
covered by a trivial fibered solid torus.  Similarly, a \emph{fibered solid
Klein bottle} is a solid Klein bottle which is finitely covered by a trivial fibered
solid torus.

A \emph{Seifert fiber space} is a 3-manifold with a decomposition into disjoint
circles, called fibers, such that each circle has a neighborhood which is a
union of fibers and is isomorphic to a fibered solid torus or a fibered Klein
bottle.

Given a Seifert fiber space $M$, one can obtain an orbifold $B$ by
quotienting out by the $S^1$-action on the fibers of $M$; that is, by
identifying each fiber to a point. By considering the quotient of neighborhoods
of fibers in $M$, the topology $B$ inherits makes it a surface with a natural
orbifold structure; we call $B$ the \emph{base orbifold} of $M$. Such an orbifold $B$ has its \emph{orbifold fundamental group}, which is not necessarily the fundamental group of the underlying topological space, but is related to the fundamental group of $M$ via the following lemma.

\begin{lemma}\cite[Lemma 3.2]{Sc83}
\label{orbifold SES}
Let $M$ be a Seifert fiber space with base orbifold $B$.  Let $\Gamma$ be the
fundamental group of $M$, and let $\Gamma_0$ be the \emph{orbifold fundamental group}
of $B$. Then there is an exact sequence \[1\to K \to \Gamma \to \Gamma_0 \to
1,\] where $K$ denotes the cyclic subgroup of $\Gamma$ generated by a regular
fiber.  The group $K$ is infinite except in cases where $M$ is covered by $S^3$.
\end{lemma}

Recall that an orbifold is called \emph{good} if it is the quotient of a manifold by an action of a discrete group of isometries. An orbifold
that is not good is called \emph{bad}.

It is known that every good $2$-orbifold is isomorphic, as an orbifold, to the quotient of $S^2$, $\dbE^2$, or $\dbH^2$  by some discrete subgroup of isometries. Hence  all closed good $2$-orbifolds can be classified as spherical, euclidean or flat, and hyperbolic. Bad $2$-orbifolds are classified in \cite[Theorem 2.3]{Sc83}.

\subsection{Geometric 3-manifolds}
A \emph{Riemannian} manifold $X$ is a smooth
manifold that admits a Riemannian metric.  If the isometry group $\isom(X)$
acts transitively, we say $X$ is \emph{homogeneous}.  If in addition $X$ has a quotient of finite
volume, $X$ is \emph{unimodular}.  A \emph{geometry} is a simply-connected,
homogeneous, unimodular Riemannian manifold along with its isometry group.  Two
geometries $(X,\isom(X))$ and $(X',\isom(X'))$ are \emph{equivalent} if
$\isom(X)\cong\isom(X')$ and there exists a diffeomorphism $X\to X'$ that
respects the $\isom(X), \isom(X')$ actions.  A geometry $(X,\isom(X))$ (often
abbreviated $X$) is \emph{maximal} if there is no Riemannian metric on $X$ with
respect to which the isometry group strictly contains $\isom(X)$.  A manifold
$M$ is called \emph{geometric} if there is a geometry $X$ and discrete subgroup
$\Gamma\leq\isom(X)$ with free $\Gamma$-action on $X$ such that $M$ is
diffeomorphic to the quotient $X/\Gamma$; we also say that $M$ \emph{admits a
geometric structure} modeled on $X$.  Similarly, a manifold with nonempty
boundary is geometric if its interior is geometric.

It is a consequence of the uniformization theorem that compact surfaces
(2-manifolds) admit Riemannian metrics with constant curvature; that is, compact
surfaces admit geometric structures modeled on $\mbs^2$, $\mbe^2$, or $\hyp$. 
In dimension three, we are not guaranteed constant curvature.  Thurston
demonstrated that there are eight $3$-dimensional maximal geometries up to
equivalence (\cite[Theorem 5.1]{Sc83}): $\mbs^3$, $\mbe^3$, $\hypp$,
$\mbs^2\times\mbe$, $\hxr$, $\pslt$, $\nil$, and $\sol$.

\subsection{Prime and JSJ decomposition}
A \emph{closed} $n$-manifold is an $n$-manifold that is compact with empty
boundary.  A \emph{connected sum} of two $n$-manifolds $M$ and $N$, denoted $M\#
N$, is a manifold created by removing the interiors of a smooth $n$-disc $D^n$
from each manifold, then identifying the boundaries $\mbs^{n-1}$.  An $n$-manifold is
\emph{nontrivial} if it is not homeomorphic to $\mbs^n$.  A \emph{prime}
$n$-manifold is a nontrivial manifold that cannot be decomposed as a connected
sum of two nontrivial $n$-manifolds; that is, $M=N\# P$ for some $n$-manifolds
$N,P$ forces either $N=\mbs^n$ or $P=\mbs^n$.  An $n$-manifold $M$ is called
\emph{irreducible} if every 2-sphere $\mbs^2\subset M$ bounds a ball $D^3\subset M$. It is well-known that all orientable prime manifolds are irreducible with the exception of $S^1\times
S^2$. The following is a well-known theorem of Kneser (existence) and Milnor (uniqueness).

\begin{theorem}[Prime decomposition]
\label{prime decomposition}
Let $M$ be a closed oriented nontrivial 3-manifold.  Then $M=P_1\#\ldots\#
P_n$ where each $P_i$ is prime.  Furthermore, this decomposition is unique up to
order and homeomorphism.
\end{theorem}

Another well known result we will need is the Jaco--Shalen--Johannson decomposition, after Perelman's work.

\begin{theorem}[JSJ decomposition]
\label{jsj decomposition}
For a closed, prime, oriented 3-manifold $M$ there exists a collection
$T\subseteq M$ of disjoint incompressible tori, i.e. two sided properly embedded and $\pi_1$-injective, such that each component of
$M\setminus T$ is either a hyperbolic or a Seifert fibered manifold.  A minimal
such collection $T$ is unique up to isotopy.
\end{theorem}

\begin{remark}\label{prime:geometric:splittings}
Note that the prime decomposition provides a graph of groups with trivial edge groups and vertex groups 
isomorphic to the fundamental group of the $P_i$'s. The fundamental group of the graph of groups will 
be isomorphic to $\pi_1(M)$. Similarly the JSJ decomposition of a prime $3$-manifold $M$ gives rise to 
a graph of groups, with all edge groups isomorphic to $\dbZ^2$, and vertex groups isomorphic to the 
fundamental groups of the Seifert fibered and hyperbolic pieces. Again, the fundamental group of the 
graph of groups will be isomorphic to $\pi_1(M)$. Each graph of groups provide a splitting for the 
fundamental groups of the initial manifold. These splittings will be used to provide reductions of the 
general computation to some special cases.
\end{remark}


\section{Push-out constructions for classifying spaces}\label{sec-push-out}

In this section we will review some push-out constructions, used to construct new classifying spaces out of old (or known) ones.

\begin{definition}
Let $\Gamma$ be any finitely generated group, and $\calF\subset \calF'$ a pair of families of subgroups of $\Gamma$. We say a collection $\calA=\{A_\alpha \}_{\alpha \in I}$ of subgroups of $\Gamma$ is adapted  to the pair $(\calF, \calF')$ provided that the following conditions hold:
\begin{enumerate}
    \item For all $A,B \in \calA$, either $A=B$ or $A\cap B\in \calF$;
    \item The collection $\calA$ is conjugacy closed;
    \item Every $A\in \calA$ is self normalizing; i.e. $N_\Gamma (A)=A$;
    \item For all $A\in \calF'\setminus \calF$, there is a $B\in \calA$ such that $A\leq B$.
\end{enumerate}
\end{definition}

\begin{proposition}\cite[p.~302]{LO09}\label{LO-push-out}
Let $\calF \subset \calF'$ be families of subgroup of $\Gamma$. Assume that the collection of subgroups $\calA=\{ H_\alpha \}_{\alpha \in I}$ is adapted to the pair $(\calF, \calF')$. Let $\calH$ be a complete set of representatives of the conjugacy classes within $\calA$, and consider the cellular $\Gamma$-push-out

\[\xymatrix{\coprod_{H\in\calH}\Gamma\times_{H}E_{\calF}H \ar[d] \ar[r]  & E_{\calF}\Gamma\ar[d]\\
\coprod_{H\in\calH}\Gamma\times_{H}E_{\calF'}H\ar[r] & X
}
\]
Then $X$ is a model for $E_{\calF'}\Gamma$. In the above cellular $\Gamma$-push-out we require either (1) the left vertical map is the disjoint union of cellular $H$-maps, and the upper horizontal is an inclusion of $\Gamma$-CW-complexes, or (2) the right vertical map is the disjoint union of inclusion of $H$-CW-complexes, and the upper horizontal map is a cellular $\Gamma$-map.
\end{proposition}

The following lemmas are straightforward consequences of the definition of adapted collections.

\begin{lemma}\label{adapted:lemma1}
Let $\Gamma$ be a finitely generated discrete group, and let $\calF\subseteq \calF' \subseteq \calF''$ be three nested families of subgroups of $\Gamma$. Let $\calA$ be a collection adapted to the pair $(\calF, \calF'')$. Then $\calA$ is adapted to the pairs $(\calF, \calF')$ and $(\calF', \calF'')$
\end{lemma}

\begin{lemma}\label{adapted:lemma2}
Let $\varphi \colon \Gamma \to \Gamma_0$ be surjective group homomorphism of discrete groups, let $\calF \subseteq \calF'$ be  pair of families of subgroups of $\Gamma_0$, and let $\calA=\{A_\alpha\}_{\alpha \in I}$ be a collection adapted to the pair $(\calF,\calF')$. Then $\tilde{\calA}=\{\varphi^{-1}(A_\alpha)\}_{\alpha\in I}$ is a collection adapted to the pair $(\tilde{\calF},\tilde{\calF}')$ of families of subgroups of $\Gamma$.
\end{lemma}

\begin{theorem}
Let $\calF$ be a family of subgroups of the finitely generated discrete group $\Gamma$. Let $\varphi \colon \Gamma \to \Gamma_0$ be a surjective homomorphism. Let $\calF_0\subseteq \calF_0'$ be a nested pair of families of subgroups of $\Gamma_0$ satisfying $\tilde{\calF}_0 \subseteq \calF \subseteq \tilde{\calF}'_0$, and let $\calA=\{A_\alpha\}_{\alpha \in I}$ be  a collection adapted to the pair $\calF_0 \subseteq \calF_0'$. Let $\calH$ be a complete set of representatives of the conjugacy classes within  $\tilde{\calA}=\{\varphi^{-1}(A_\alpha)\}_{\alpha\in I}$, and consider the following cellular $\Gamma$-push-out
\[\xymatrix{\coprod_{\tilde{H}\in\calH}\Gamma\times_{\tilde{H}}E_{\calF_0}H \ar[d] \ar[r]  & E_{\calF_0}\Gamma_0\ar[d]\\
\coprod_{\tilde{H}\in\calH}\Gamma\times_{\tilde{H}}E_{\calF}\tilde{H}\ar[r] & X
}
\]
Then $X$ is a model for $E_{\calF}\Gamma$. In the above cellular $\Gamma$-push-out we require either (1) the left vertical map is the disjoint union of cellular $\tilde{H}$-maps, and the upper horizontal is an inclusion of $\Gamma$-CW-complexes, or (2) the right vertical map is the disjoint union of inclusion of $\tilde{H}$-CW-complexes, and the upper horizontal map is a cellular $\Gamma$-map.
\end{theorem}
\begin{proof}
It can be easily verified that, via restriction with $\varphi$, $E_{\calF_0}\Gamma_0=E_{\tilde{\calF}_0}\Gamma$ and $E_{\calF_0}H=E_{\tilde{\calF}_0}\tilde{H}$. From Lemmas \ref{adapted:lemma1} and \ref{adapted:lemma2} we have that $\tilde{\calA}$ is a collection adapted to the pair $(\tilde{\calF}_0,\calF)$. Then by Proposition \ref{LO-push-out} we have that the above push-out is a model for $E_\calF \Gamma$.
\end{proof}

The following immediate corollary is more suitable for our purposes, and it will be used jointly with Lemma \ref{orbifold SES}.

\begin{corollary}\label{push-out:cyclic:kernel}
Let $\Gamma$ be  finitely generated discrete group. Let $\varphi \colon \Gamma \to \Gamma_0$ be a surjective homomorphism with cyclic kernel. Let $\fin_0$ and $\vcyc_0$ be the famlies of finite and virtually cyclic subgroups of  $\Gamma_0$ respectively. Let $\calA$ be  a collection adapted to the pair $(\fin_0, \vcyc_0)$. Let $\calH$ be a complete set of representatives of the conjugacy classes within  $\tilde{\calA}$, and consider the following cellular $\Gamma$-push-out
\[\xymatrix{\coprod_{\tilde{H}\in\calH}\Gamma\times_{\tilde{H}}\underline{E}H \ar[d] \ar[r]  & \underline{E}\Gamma_0\ar[d]\\
\coprod_{\tilde{H}\in\calH}\Gamma\times_{\tilde{H}}\underline{\underline{E}}\tilde{H}\ar[r] & X
}
\]
Then $X$ is a model for $\underline{\underline{E}}\Gamma$. In the above cellular $\Gamma$-push-out we require either (1) the left vertical map is the disjoint union of cellular $\tilde{H}$-maps, and the upper horizontal is an inclusion of $\Gamma$-CW-complexes, or (2) the right vertical map is the disjoint union of inclusion of $\tilde{H}$-CW-complexes, and the upper horizontal map is a cellular $\Gamma$-map.
\end{corollary}

Next, we will show how to construct a classifying space (for a suitable family) of the fundamental group of a graph of groups, by assembling the classifying spaces of the vertex groups and the edge groups. We will later apply these constructions in conjunction with Remark \ref{prime:geometric:splittings}.\\

Let $\mathbf{Y}$ be a graph of groups with vertex groups $Y_P$ and edge groups $Y_y$, and fundamental group $G$. We are going to construct a graph of spaces $\mathbf{X}$ using the classifying spaces of the edges and the vertices and the corresponding families of virtually cyclic subgroups. Let $X_{P}$ be a model for
for the classifying space $\evc Y_P$, and $X_{y}$ be a model for $\evc Y_y$, for every vertex $P$ and edge $y$ of $Y$. So for every monomorphism $Y_y\to Y_{t(y)}$ we have a $Y_y$-equivariant cellular map (unique up to $Y_y$-homotopy) $X_y \to X_{t(y)}$ which leads to the $G$-equivariant cellular map $G\times_{Y_y} X_y \to G\times_{Y_{t(y)}} X_{t(y)} $. This gives us the information required to define a graph of spaces $\mathbf{X}$ with underlying graph $T$, the Bass-Serre tree of $\mathbf{Y}$. Moreover, we have a cellular $G$-action on the geometric realization $X$ of $\mathbf{X}$.


\begin{proposition}
\label{bass serre construction}
The geometric realization $X$ of the graph of spaces $\mathbf{X}$ constructed above is a model for $E_{\fmly}G$, where $\fmly$ is the family of virtually
cyclic subgroups of $G$ that are conjugate to a virtually cyclic subgroup in one
of the $Y_y$ or $Y_P$, $y\in \edge Y$, $P\in \vertex Y$. In particular, there exists a model for $E_\calF G$ of dimension  \[\gd_\calF(G)\leq \max\{\gdvc(Y_y)+1,\gdvc(Y_P)| y\in \edge Y, P\in \vertex Y \}. \]
\end{proposition}
\begin{proof}
Let $\phi:X\to T$ be a deformation retraction collapsing each copy of $X_{P}$
down onto the vertex $\widetilde{P}\in\vrt T$ to which it corresponds, and similarly
collapsing each $X_{y}\times[0,1]$ down along the $X_{y}$ component to the
corresponding edge $\widetilde{y}\in\edge T$, $\widetilde{y}\simeq [0,1]$. Then  $G$ has a natural action on $X$ via left multiplication, which permutes the
copies of each $X_{P}$ or $X_{y}$ so that $\phi$ is a $G$-equivariant map.
It remains to show that the $G$-$CW$-complex $X$ is a model for $E_{\fmly}G$.

Suppose first that $H\leq G$ is not in $\fmly$; then either $H$ is not
conjugate into any vertex subgroup or is not virtually cyclic.  If $H$ is not
conjugate into any vertex subgroup, then $T^H=\emptyset$; in particular, $H$
does not fix any copy of $X_{P}$ or $X_{y}\times[0,1]$ in $X$, so
$X^H=\emptyset$.  If $H$ is not virtually cyclic, then even if the $H$-action on
$T$ does fix some nonempty subgraph, the $H$-action on any corresponding
$X_{P}$ or $X_{y}\times[0,1]$ must have empty fixed set; again, this
implies that $X^H=\emptyset$.

On the other hand, suppose $H\in\fmly$; that is, $H$ is virtually cyclic and
conjugate into the subgroup $Y_P$, for some vertex $P$.  Then $H$ fixes the copy of
$X_{P}$ in $X$ corresponding to a fixed vertex in $T$. As $H$ is virtually
cyclic, and $X_{P}$ is a model for $\evc Y_{P}$, $(X_{P})^H$ is not empty. 
Moreover, given any two vertices $P$ and $Q$ in $\vrt T$ fixed by the
$H$-action, the unique geodesic path $c$ in $T$ connecting $P$ and $Q$ must also
be fixed; in particular, $T^H$ is a connected subgraph of the tree $T$, so that
$T^H$ is itself a tree.  Let $\widetilde{y}\in\edge T$ be fixed, and consider
the copy $X_{y}\times[0,1]$ with $\phi$-image $\widetilde{y}$; then the
$H$-action on $X_{y}$ has nonempty fixed set $(X_{y})^H$, so in particular
$(X_{y})^H\times[0,1]$ is nonempty.  Thus, the $\phi$-preimage of $T^H$ is a
nonempty, connected subspace $X^H\subseteq X$.  To see that $X^H$ is
contractible, first contract down along $\phi$ to $T^H$, then contract the tree
$T^H$.

\end{proof}

\begin{definition}\label{acylindrical:def}
Let  $\mathbf{Y}$ be a graph of groups with fundamental group $G$. The splitting of $G$ is said to be \emph{acylindrical} if there exists an integer $k$ such that, for every path $c$ of length $k$ in the Bass-Serre tree $T$ of $\mathbf{Y}$, the stabilizer of $c$ is finite.
\end{definition}

The following proposition roughly says that, provided $\mathbf{Y}$ gives an ayclindrical splitting of $G$, you can attach $2$-cells to the classifying space from Proposition \ref{bass serre construction} to get a model for $\evc G$. 

\begin{proposition}\label{push-out:acyl}
Let $\mathbf{Y}$ be a graph of groups giving an acylindrical splitting of $G$. Let $\calF$ be the family of virtually cyclic subgroups of $G$ that conjugate into a vertex group in $\mathbf{Y}$. Let $\calA$ be the collection of maximal virtually cyclic subgroups of $G$ not in $\calF$. Let $\calH$ be a complete set of representatives of the conjugacy classes within $\calA$. Let $\{*\}$ be the one point space, and consider the following cellular $G$-push-out
\[\xymatrix{\coprod_{H\in\calH} G\times_{H}\dbE \ar[d] \ar[r]  & E_{\calF}G\ar[d]\\
\coprod_{H\in\calH}G\times_{H}\{*\}\ar[r] & X
}
\]
Then $X$ is a model for $\underline{\underline{E}}G$. In the above cellular $G$-push-out we require either (1) the left vertical map is the disjoint union of cellular $H$-maps, and the upper horizontal is an inclusion of $G$-CW-complexes, or (2) the right vertical map is the disjoint union of inclusion of $H$-CW-complexes, and the upper horizontal map is a cellular $G$-map.
\end{proposition}

\begin{proof}
From \cite[Claim 3]{LO09_2} we know that $\calA$ is an adapted collection to the pair $(\calF, \vcyc)$. Hence by Proposition \ref{LO-push-out} we just have to prove that, for all $H\in \calH$, $\dbE$ and $\{*\}$ are models for $E_\calF H$ and $\evc H$ respectively. 
It is clear that $\{*\}$ is a model for $\evc H$ since $H$ is virtually cyclic. To verify the other case, it suffices to check that $\calF\cap H$ is the family of finite subgroups of $H$. 

Let $V\in\fmly$ satisfy $V\leq H$. Then $V$ can be conjugated into $Y_P$ for
some vertex group in $Y$; say $gVg^{-1}\leq Y_P$. Since $V$ is virtually
cyclic, let $\langle v\rangle$ be a finite-index cyclic subgroup of $V$.  If $\langle
v\rangle$ is infinite, then $[H:\langle v\rangle]<\infty$.  If $h\in
H\setminus\langle v\rangle$, then $h^n\in\langle v\rangle$ for some
$n\in\mbn$, which means that $gh^ng^{-1}\in Y_P$. Now recall that $G$ can be identified with a subgroup of the free product $F(\mathbf{Y})=*_{\vertex Y}Y_P *F_{\edge Y}$, where $F_{\edge Y}$ is the free group on $\edge Y$.  Therefore there is no element $f\in F(\mathbf{Y})\setminus Y_P$ with a power in $Y_P$.  Thus, we
must have that $ghg^{-1}\in Y_P$.  But this implies that $gHg^{-1}\leq Y_P$,
contradicting that $H\notin\fmly$. So we conclude that $\langle v\rangle$ must be finite.  Since $V$ has a finite subgroup of finite index, it is also finite.

Conversely, if $F\leq H$ is a finite subgroup, then $T^F$ is
nonempty (note that $T$ is $\cat{0}$). 
In particular, $F$ fixes some vertex $gY_P\in\vrt T$, and is therefore conjugate
to a subgroup of $Y_P$. This shows $F\in\fmly$, giving the reverse containment.

\end{proof}

\begin{corollary}\label{vcgd:acyl}
Let $\mathbf{Y}$ be a graph of groups giving an acylindrical splitting of $G$. Then
\[
\gdvc(G)\leq \max\{2, \gdvc(Y_y)+1,\gdvc(Y_P)| y\in \edge Y, P\in \vertex Y\}.
\]
\end{corollary}
\begin{proof}
Fix minimal models for $\evc Y_y$, and $\evc Y_P$ for every edge $y$ and every vertex $P$ of $Y$. Now from Proposition \ref{bass serre construction} we get a model $X$  for $E_\calF G$ of dimension $\max\{\gdvc(Y_y)+1,\gdvc(Y_P)| y\in \edge Y, P\in \vertex Y \}$. Now using Proposition \ref{push-out:acyl} we can attach $2$-cells to $X$ in order to obtain a model for $\evc G$ of dimension $\max\{2,\gdvc(Y_y)+1,\gdvc(Y_P)| y\in \edge Y, P\in \vertex Y \}$, therefore this is an upper bound for $\gdvc(G)$.
\end{proof}

As a quick application, we now explain how to reduce our analysis of the virtually cyclic geometric dimension of a $3$-manifold group to the prime case.

\begin{theorem}\label{reduction:irreducible}
Let $M$ be a closed, orientable, connected $3$-manifold. Consider the prime decomposition $M=P_1\# \cdots \# P_k$. Denote $\Gamma=\pi_1(M)$, $\Gamma_i=\pi_1(P_i)$. Then
\[
\max\{\gdvc(\Gamma_i)|1\leq i\leq k\}\leq \gdvc(\Gamma) \leq \max\{2, \gdvc(\Gamma_i) | 1\leq i\leq k\}.
\]
\end{theorem}
\begin{proof}
Since each $\Gamma_i$ is a subgroup of $\Gamma$, the first inequality comes from Lemma \ref{properties:gd}.
On the other hand, the prime decomposition of $M$ determines a graph of spaces with fundamental group $\Gamma$ and hence a graph of groups $\mathbf{Y}$ (see Remark \ref{prime:geometric:splittings}). Since the edge groups are all isomorphic to $\pi_1(S^2)=1$, stabilizers of the edges in the Bass-Serre tree $T$ of $\mathbf{Y}$ are trivial. Thus the splitting of $\Gamma$ is acylindrical (with $k=1$ in Definition \ref{acylindrical:def}). The conclusion now follows from Corollary \ref{vcgd:acyl}. 
\end{proof}

The reader might naturally wonder whether a similar reduction can be performed with the JSJ decomposition. This is indeed the
case, but acylindricity of the splitting is much more subtle in that case (and does not always hold). We will discuss this is detail in 
Section \ref{sec-reduction}.


\section{The Seifert fibered case}\label{sec-Seifert-fibered}

In this section, we will study the geometric dimension of the fundamental groups of compact Seifert
fibered manifolds, both in the case where they have toral boundary components (e.g. pieces in the
JSJ decomposition of a prime $3$-manifold) and the case where they have no boundary 
(e.g. are themselves prime $3$-manifolds). 

\subsection{Seifert fibered manifolds without boundary}

Using the base orbifold $B$ of a Seifert fibered manifold, we have the following classification

\begin{itemize}
    \item $B$ is a bad orbifold;
    \item $B$ is a good orbifold, modeled on either $S^2$, $\dbH^2$ or $\dbE^2$.
\end{itemize}

The following proposition deals with the case where $B$ is a bad manifold, or is a good manifold modeled on $S^2$.

\begin{proposition}\label{bad:orb:case}
Let $M$ be a closed Seifert fiber space with base orbifold $B$ and fundamental
group $\Gamma$. Assume that $B$ is either a bad orbifold, or a good orbifold modeled on $S^2$. Then $\Gamma$ is virtually cyclic. In particular, $\gdvcgamma=0$
\end{proposition}

\begin{proof}
Suppose first that $B$ is modeled on $S^2$.  Then
$\Gamma_0=\pi_1(B)$ is a discrete subgroup of $SO(3)\cong\isom(S^2)$ and is
therefore finite. By the short exact sequence given in Lemma \ref{orbifold SES},
$\Gamma$ is virtually cyclic.

From the classification of bad manifolds (see \cite[Theorem~2.3]{Sc83}) we know that if $B$ is a bad manifold then its  orbifold Euler characteristic is positive. So by \cite[Theorem 5.3(ii)]{Sc83} $M$ is modeled on one of
$S^3$ or $S^2\times\mbe$. Now we conclude by observing that discrete subgroups in either $\isom(S^3) \cong SO(4)$ or
$\isom(S^2\times\mbe) \cong SO(3)\times (\mbr\rtimes\mbz_2)$ are virtually cyclic.
\end{proof}

Now we have to deal with the case where $B$ is a good orbifold modeled on $\dbH^2$ or $\dbE^2$.

\begin{proposition}
\label{hyp B model}
Let $M$ be a closed Seifert fiber space with base orbifold $B$ modeled on
$\hyp$.  Let $\Gamma=\pi_1(M)$ and $\Gamma_0=\pi_1(B)$ be the respective
fundamental groups.  Let $\mc{A}$ be the collection of maximal infinite virtually cyclic
subgroups of $\Gamma_0$, let $\widetilde{\mc{A}}$ be the collection of preimages
of $\mc{A}$ in $\Gamma$, and let $\mc{H}$ be a set of representatives of
conjugacy classes in $\widetilde{\mc{A}}$. Consider the following
cellular $\Gamma$-push-out:

\[\xymatrix{\coprod_{\tilde{H}\in\calH}\Gamma\times_{\tilde{H}}\dbE \ar[d] \ar[r]  & \dbH^2 \ar[d]\\
\coprod_{\tilde{H}\in\calH}\Gamma\times_{\tilde{H}}\evc\tilde{H}\ar[r] & X
}
\]
Then all $\widetilde{H}\in\mc{H}$
are virtually 2-crystallographic, and $X$ is a model for $\evc\Gamma$.  In the above cellular
$\Gamma$-push-out, we require either (1) the left vertical map is the disjoint union of
cellular $\widetilde{H}$-maps ($\widetilde{H}\in\mc{H}$), the upper horizontal map is an
inclusion of $\Gamma$-CW-complexes, or (2) the left vertical map is the disjoint union of
inclusions of $\widetilde{H}$-CW-complexes ($\widetilde{H}\in\mc{H}$), the upper horizontal map
is a cellular $\Gamma$-map.
\end{proposition}

\begin{proposition}\label{empty:bdary:hyp:orb}
Let $M$ be a closed Seifert fibered manifold with base orbifold modeled on $\dbH^2$, 
and let $\Gamma=\pi_1(M)$ be the fundamental group of $M$. Then $\gdvc(\Gamma)=3$.
\end{proposition}

\begin{proof}[Proof of Proposition \ref{hyp B model}]
We have the 
short exact sequence 
\[1\to\dbZ\to \Gamma \to \Gamma_0 \to 1\]
from Lemma \ref{orbifold SES}. Then $\Gamma_0$ is a lattice in $\mathrm{Isom}(\dbH^2)\cong \mathrm{PSL}_2(\dbR)\rtimes \dbZ/2$,
hence is hyperbolic.
Let $\vcyc_0$ and $\fin_0$ be the families of virtually cyclic subgroups and finite 
subgroups of $\Gamma_0$ respectively. Applying \cite[Theorem 2.6]{LO07}, we see that the collection $\mc{A}$ of \emph{maximal}
infinite virtually cyclic subgroups of $\Gamma_0$ is adapted to the pair
$(\fin_0,\vcyc_0)$.

This allows us
to apply the construction of Corollary \ref{push-out:cyclic:kernel}.  Since $\Gamma_0$ is a lattice in $\mathrm{Iso(\hyp)}$, we have that
$\hyp$ is a model for $\efin\Gamma_0$.  Let $H\in\mc{A}$. Therefore $H$ is the stabilizer of a unique
geodesic $c$ in $\hyp$.  The subgroup of $H$ that fixes $c$ is finite and
normal; the quotient of $H$ by this group inherits an effective action on
$c\simeq\mbe$, and is therefore 1-crystallographic. This gives $\mbe$ as a model
for $\efin H$. 

Consider $\widetilde{H}\in\mc{H}$,  $H=\phi(\widetilde{H})$ and its action on $\hyp$ (on which $B$ is
modeled).  As $H$ is virtually cyclic, we know it stabilizes a unique
geodesic $c$, hence $H$ has a natural action on $\mbe$ as $c\simeq\mbe$. We now consider the preimage of this copy of $\mbe$ in
the lift of the Seifert fiber space $M$ to its universal cover.  By
Lemma \ref{orbifold SES}, $\Gamma\to\Gamma_0$ is infinite cyclic unless $M$ is
modeled on $S^3$, but this will contradict \cite[Theorem~5.3]{Sc83};
lifting $\mbe$ to $\widetilde{M}$ we then get a $\widetilde{H}$-action on
$\mbe^2$. Let $\widetilde{F}\leq\widetilde{H}$ be the subgroup with trivial
action on $\mbe^2$. Then $F=\phi(\widetilde{F})$ is a subgroup $F\leq H$, which cannot contain the hyperbolic element that generates the finite-index
infinite cyclic subgroup of $H$, so must be finite.  Since $\ker\phi\cong\mbz$ acts non-trivially on the fiber direction, $\widetilde{F}\cong F$ must also be finite. Letting
$Q=\widetilde{H}/\widetilde{F}$, we get that $Q$ is 2-crystallographic, 
as it
inherits an effective cocompact action on $\mbe^2$. In particular the model for $\evc Q$ given by \cite{CFH06} will provide a $3$-dimensional model for $\evc \tilde{H}$.
\end{proof}

\begin{proof}[Proof of Proposition \ref{empty:bdary:hyp:orb}]
The model constructed is three dimensional, as $\evc\widetilde{H}$ and
$\mbe^2\times [0,1]$ are both three dimensional; this gives that
$\gdvc(\Gamma)\leq 3$. Since $\Gamma$ has a subgroup isomorphic
to $\mbz^2$ (consider any $\widetilde{H}\in\mc{H}$), the result follows from Lemma \ref{properties:gd}.

\end{proof}

\begin{proposition}\label{empty:bdary:flat:orb}
Let $M$ be a closed Seifert fibered manifold with base orbifold modeled on $\dbE^2$, and let $\Gamma=\pi_1(M)$ be the fundamental group of $M$. Then 
\begin{itemize}
    \item $M$ is modeled on $\dbE^3$, and $\gdvc(\Gamma)=4$; or
    \item $M$ is modeled on $\nil$, and $\gdvc(\Gamma)=3$.
\end{itemize}
\end{proposition}
\begin{proof}
From \cite[Theorem 5.3(ii)]{Sc83} we know that $M$ is modeled either on $\dbE^3$ 
or on $\nil$. In the former case we have that $\Gamma$ is $3$-crystallographic, 
and by \cite{CFH06} we have that $\gdvcgamma=4$.

For the case where $M$ is modeled on $\nil$, we would like to use \cite[Theorem~5.13]{LW12}. 
For this we will first prove that $\Gamma$ is virtually poly-$\dbZ$, and then check that 
$\Gamma$ satisfies \cite[Theorem~5.13~case~2b]{LW12}, so that $\gdvcgamma=3$. The two 
conditions to check are:

\begin{enumerate}
\item There is an infinite normal subgroup $C\subseteq\Gamma$, and for every infinite cyclic 
subgroup $D\subseteq \Gamma$ with $[\Gamma: N_\Gamma D]<\infty $ we have $C\cap D\neq 1$.
\item There exists no subgroup $W\leq \Gamma$ such that its commensurator $N_\Gamma[W]$ has 
virtual cohomological dimension equal to $1$.
\end{enumerate}
From Lemma \ref{orbifold SES} we have the short exact sequence
\[
1\to \dbZ \to \Gamma \to \Gamma_0 \to 1
\]
where $\Gamma_0$ is the orbifold fundamental group of $B$, which is $2$-crystallographic by 
hypotheses, in particular it is virtually poly-$\dbZ$ with a filtration of the form 
$1\subseteq \dbZ \subseteq \dbZ^2 \subseteq \Gamma_0$. Since the property of being virtually 
poly-$\dbZ$ is closed under taking extensions (see \cite[Lemma~5.14, i-iv]{LW12}) we conclude 
that $\Gamma$ is virtually poly-$\dbZ$.

$\nil$ is the continuous Heisenberg Lie group, and can be
identified with the group $\mbr^3$ with multiplication given by
$(a,b,c)\cdot(d,e,f)=(a+d,b+e,c+f-ae)$. The center of $\nil$ is the subgroup
$\{(0,0,z): z\in\mbr\}$.  It is a simple matter to compute conjugates and
positive powers:
\[(x,y,z)^{-1}\cdot (a,b,c)\cdot (x,y,z)=(a, b, ay-bx+c);\]
\[(a,b,c)^n = \left(na, nb, nc-\frac{n(n-1)}{2}ab\right)\textrm{ for }n>0.\]

\vskip 10pt

We now verify property (1). Let  $C\subseteq\Gamma$ be
the center $Z(\Gamma)$ of $\Gamma$. We first point out that $C$ is infinite cyclic. Indeed, the group $\Gamma$ can be
viewed as a lattice in $\nil$, which implies $C=\Gamma\cap Z(\nil)$ is a discrete cocompact subgroup in $Z(\nil) \cong \dbR$.

Let $D$ be an infinite cyclic subgroup of $\Gamma$, generated by the element
$(a,b,c)\in\nil$. If $(x,y,z)\in N_{\Gamma}D$,
then we must have for some $n$:
\[(x,y,z)^{-1}(a,b,c)(x,y,z)=(a,b,c)^n.\]
So if $(x,y,z)\in N_{\Gamma}D$ then from the formulas above we see that  $n=1$ and $ay-bx=0$.

\vskip 5pt

Without loss of generality suppose $a\neq 0$.  Then for
$(x,y,z)$ to be a normalizer of $D$, we need $ay-bx=0$, or $y=\frac{b}{a}x$. 
Thus, we can consider the closed Lie subgroup 
\[H:=\left\{\left(\alpha,\frac{b}{a}\alpha,\beta\right)\ \bigg|\
\alpha,\beta\in\mbr\right\}\]
of $\nil$, and observe that $H$ is isomorphic to $\mathbb{R}^2$. Since $N_{\Gamma}D=H\cap\Gamma$, 
we see that $N_{\Gamma}D$ can be viewed as a discrete subgroup of 
$H\cong \mathbb{R}^2$, which forces $\vcd(N_{\Gamma}D)\leq 2$.
Since $\vcd(\Gamma)=3$, we conclude that $[\Gamma:N_{\Gamma}D]=\infty$.  Thus, for
any infinite cyclic $D\leq\Gamma$, either $Z(\Gamma)\cap D\neq\{1\}$ or
$[\Gamma:N_{\Gamma}D]=\infty$; in particular, $\Gamma$ satisfies
condition (1) above.

Finally, since $\vcd(\Gamma)=3$, and $N\leq N_\Gamma D\leq N_\Gamma[D]$ and $\vcd (N)=2$ we verify condition (2) above, completing the proof.
\end{proof}


\subsection{Seifert fibered manifolds with boundary}
In this section we study compact Seifert fibered manifolds with non-empty boundary. Throughout this section, $M$ will always be a compact Seifert fibered 
manifold with nonempty boundary. Let $\Gamma=\pi_1(M)$ be the fundamental group, 
let $\Gamma_0=\pi_1^\mathrm{orb}(B)$ be the orbifold fundamental group of the 
base orbifold $B$, and let $\phi:\Gamma\to \Gamma_0$ be the associated homomorphism.

First we will see that we do not have to consider the case of $B$ being a bad orbifold or a good orbifold modeled on $S^2$.

\begin{lemma}\label{impossible:compact:case}
Let $M$ be a compact Seifert fibered manifold with nonempty boundary with base orbifold $B$. Denote by $B^\circ$ be the interior of $B$. Then $B^{\circ}$ is a good orbifold modeled either on $\dbH^2$ or on $\dbE^2$.
\end{lemma}

\begin{proof}
By the classification of bad orbifolds (see \cite[Theorem~2.3]{Sc83}), the only 
bad orbifolds without boundary
have compact underlying space, so $B^{\circ}$ must be good, and therefore
finitely covered by a 2-manifold $N^{\circ}$ that is also not compact. Then
$N^{\circ}$ is geometric; by the uniformization theorem, all geometric surfaces
are modeled on $\mbs^2$, $\mbe^2$, or $\hyp$. Since  $\mbs^2$ is compact, 
all quotients by discrete (finite) actions are also compact, therefore $N^\circ$ 
cannot be modeled on this geometry. The lemma follows.
\end{proof}

\begin{proposition}\label{prop:bdry:flat:orb}
Let $M$ be a compact Seifert fibered manifold with nonempty boundary. Let $\Gamma=\pi_1(M)$, and let $B$ be the base orbifold of $M$. If $B^\circ$ is modeled on $\dbE^2$, then $\Gamma$ is $2$-crystallographic isomorphic to $\dbZ^2$ or $\dbZ\rtimes \dbZ$. In particular $\gdvc(\Gamma)=3$.
\end{proposition}

\begin{proof}
Then by \cite[Theorem~1.2.2]{Mo05} we have that $M$ is modeled on $\dbE^3$ or $\nil$,
or is diffeomorphic to $S^1\times S^1\times I$, or is a twisted $[0,1]$-bundle over
the Klein bottle. But neither $\dbE^3$ nor $\nil$ admit noncompact geometric quotients.
In the remaining two cases the fundamental group of $M$ is isomorphic to $\dbZ^2$ or
to $\dbZ\rtimes \dbZ$ respectively. Hence $\Gamma$ is a $2$-crystallographic group,
and the conclusion follows from Lemma \ref{properties:gd}.
\end{proof}

\begin{proposition} \label{prop:bdry:hyp:orb}
Let $M$ be a compact Seifert fiber space with nonempty boundary, and let
$\Gamma=\pi_1(M)$ be the fundamental group.  Suppose that the interior of the base orbifold $B^\circ$ is modeled on $\dbH^2$, and has orbifold fundamental group
$\pi_1^{\mathrm{orb}}(B)=\Gamma_0$. Let
$\phi:\Gamma\to \Gamma_0$ be the quotient map, and let $\mc{A}$ be the
collection of preimages of maximal infinite virtually cyclic subgroups of
$\Gamma_0$. Let $\mc{H}$ be a set of representatives of conjugacy classes in
$\mc{A}$.  Consider the following cellular $\Gamma$-push-out:

\[\xymatrix{\coprod_{\tilde{H}\in\calH}\Gamma\times_{\tilde{H}}\dbE \ar[d] \ar[r]  & \underline{E}\Gamma_0 \ar[d]\\
\coprod_{\tilde{H}\in\calH}\Gamma\times_{\tilde{H}}\evc\tilde{H}\ar[r] & X
}
\]

Then $X$ is a model for $\evc\Gamma$.  In the above cellular
$\Gamma$-push-out, we require either (1) the left vertical map is the disjoint union of
cellular $\widetilde{H}$-maps ($\widetilde{H}\in\mc{H}$), the upper horizontal map is an
inclusion of $\Gamma$-CW-complexes, or (2) the left vertical map is the disjoint union of
inclusions of $\widetilde{H}$-CW-complexes ($\widetilde{H}\in\mc{H}$), the upper horizontal 
map is a cellular $\Gamma$-map.

Moreover, $\underline{E}\Gamma_0$ admits a $1$-dimensional model, and $\evc \widetilde{H}$ 
admits a $3$ dimensional model. In particular   $\gdvc(\Gamma)=3$.
\end{proposition}
\begin{proof}
We would like to use Proposition \ref{LO-push-out} to construct a model for $\evc\Gamma$. 
Let $\fmly'$ be the family of virtually cyclic subgroups of $\Gamma$ and $\fmly$ be the family 
of virtually cyclic subgroups $F$ of $\Gamma$ such that $\phi(F)\leq\Gamma_0$ is finite.
In order to use Proposition \ref{LO-push-out}, we need a model for
$E_{\fmly}\Gamma$, an adapted collection $\mc{A}$, and models for $E_{\fmly}H$
and $\evc H$ for each $H\in\mc{A}$.

First, a model for $E_{\fmly}\Gamma$ is the same as a model for $\underline{E}\Gamma_0$. On 
the other hand $\Gamma_0$ contains as a finite index subgroup the fundamental group of a 
surface with non-empty boundary, therefore $\Gamma_0$ is virtually free. Hence $\Gamma_0$ 
admits a splitting as a fundamental grouph of a graph of groups of finite groups, so the 
Bass-Serre tree $T$ of such a splitting is a model for $\underline{E}\Gamma_0$.

Next let us describe an adapted collection $\calA$.

 Let $\mc{A}$ be the collection of subgroups of $\Gamma$ that are 
 preimages of maximal infinite virtually cyclic subgroups of $\Gamma_0$. 
 Then we claim that $\mc{A}$ is adapted to the pair
$(\fmly,\fmly')$ of families of subgroups of $\Gamma$. In fact, the virtually 
cyclic subgroups of $\Gamma_0$ that are conjugate into a
vertex group of the graph of groups presentation must be finite, since the vertex
groups themselves are finite. In particular, the splitting of $\Gamma_0$ given
by the graph of groups is acylindrical.  By \cite[Claim~3]{LO09_2}, the
collection $\mc{A}_0$ of maximal infinite virtually cyclic subgroups of
$\Gamma_0$ is adapted to the pair $(\mc{FIN}_0,\mc{VC}_0)$ of families of finite
and virtually cyclic subgroups of $\Gamma_0$, respectively.  By Lemma \ref{adapted:lemma2}, 
$\mc{A}=\widetilde{\mc{A}_0}$ is therefore adapted to the pair
$(\fmly,\widetilde{\mc{VC}_0})$ of families of subgroups of $\Gamma$.  Since
$\fmly\subseteq\fmly'\subseteq\widetilde{\mc{VC}_0}$, Lemma \ref{adapted:lemma1}
shows that $\mc{A}$ is adapted to the pair $(\fmly,\fmly')$, as claimed.

Let $\widetilde{H}\in\mc{A}$ be the $\phi$-preimage of a maximal infinite
virtually cyclic subgroup $H\leq\Gamma_0$. A model for
$E_{\fmly}\widetilde{H}$ is the same as a model for $\underline{E}H$. Since 
$H$ is virtually cyclic, $\dbE$ is a model for $\underline{E}H$.

It remains to construct a model for $\evc\widetilde{H}$. But this can be done by an argument identical to the one at the end of Proposition \ref{hyp B model}. We leave the details to the reader.
\end{proof}

\begin{remark}
Note that in both cases whether the base orbifold is modeled on $\dbE^2$ or $\dbH^2$ 
we have explicit models for $\evc G$. In fact, in Proposition \ref{prop:bdry:flat:orb} 
we can use the construction of \cite{CFH06}. While in Proposition \ref{prop:bdry:hyp:orb} 
we have an explicit push-out where $\underline{E}\Gamma_0$ is a tree and every $\tilde{H}$ 
has finite normal subgroup such that the quotient is crystallographic, and again we can 
use \cite{CFH06}.
\end{remark}

\begin{remark}\label{remark:bdry:hyp:orb}
Note that the adapted collection $\calA$ constructed in the proof of Proposition 
\ref{prop:bdry:hyp:orb}, consists  of preimages of maximal infinite virtually cyclic 
subgroups of $\Gamma_0$ in $\Gamma$. So $\calH$ contains representatives of the conjugacy 
classes of the boundary torus of $M$. This fact will be used in some of the
Bredon cohomology computations in Section \ref{sec:Bredon}.
\end{remark}


\section{The hyperbolic case}\label{sec-hyperbolic}

In this section, we will analyze the geometric dimension of lattices in the isometry group $PSL(2, \mathbb C)$
of hyperbolic $3$-space. Since we are going to use some standard properties of hyperbolic $3$-dimensional 
geometry we refer the reader to \cite[p.~448]{Sc83} for details about the geometry of $\hypp$.

\begin{proposition}\label{hyperbolic:case}
Let $M$ be a connected, oriented, finite-volume hyperbolic $3$-manifold, and let $\Gamma=\pi_1(M)$. 
Then $\gdvc(\Gamma)=3$.
\end{proposition}

\begin{proof}
In order to establish the proposition, we start by using a push-out construction to create a model for $\evc \Gamma$. This will
provide an upper bound on $\gdvc(\Gamma)$.

The group $\Gamma$ is a relatively hyperbolic group, relative to the collection of maximal parabolic subgroups $P_\xi$ of $\Gamma$.
From \cite[Theorem 2.6]{LO07}, we know that the collection $\mc{A}$ of infinite maximal subgroups $M_c$ that stabilize a
geodesic $c(\mbr)\subset\hypp$ and infinite maximal parabolic subgroups
$P_{\xi}$ that fix a unique boundary point $\xi\in\partial\hypp$ is adapted to
the pair $(\fin,\vcyc)$. 

Let $\mc{A}$ be the collection of infinite maximal $M_c$ or $P_{\xi}$ subgroups of $\Gamma$. Let $\mc{H}$ be a
complete set of representatives of the conjugacy classes within $\mc{A}$, and
consider the following cellular $\Gamma$-push-out:

\[\xymatrix{\coprod_{H\in\calH}\Gamma\times_{H}\underline{E}H\ar[d] \ar[r]  & \underline{E}\Gamma=\dbH^3 \ar[d]\\
\coprod_{H\in\calH}\Gamma\times_{H}\evc H\ar[r] & X
}
\]
Then Proposition \ref{LO-push-out} tells us that $X$ is a model for $\evc\Gamma$. Since $X$ is $3$-dimensional, we obtain
the inequality $\gdvc(\Gamma)\leq 3$.  If $\Gamma$ is nonuniform, it contains subgroups isomorphic to $\mbz^2$, and the 
conclusion follows from Lemma \ref{properties:gd}.

Suppose now that $\Gamma$ is a uniform lattice. The push-out construction above gives rise to the Mayer-Vietoris
sequence
\[\cdots\to
\hvc^3(\Gamma;\underline{\mbz})\to
\bigg(\bigoplus_{H\in\mc{H}}\hvc^3(H;\underline{\mbz})\bigg)\oplus \hfin^3(\Gamma;\underline{\mbz})\to
\bigoplus_{H\in\mc{H}}\hfin^2(H;\underline{\mbz})  \to \cdots\]
Note that in this case $H$ is always of the form $M_c$ (there are no $P_\xi$ elements in $\calH$ because we have no parabolic elements), 
hence it is virtually cyclic. Moreover, we have $\hfin^3(\Gamma;\underline{\mbz})\cong H^3(M;\dbZ)\cong \dbZ$. The Mayer-Vietoris sequence thus
simplifies to
\[\ldots\to \hvc^3(\Gamma;\underline{\mbz})\to \mbz\to 0\to\ldots\]
This gives the lower bound $3\leq \cdvc(\Gamma)\leq \gdvcgamma$ and completes the proof.
\end{proof}

\section{Two exceptional cases}\label{sec-twisted-bundles}

In this section, we focus on manifolds whose JSJ decomposition has all pieces that are Seifert fibered
with Euclidean base orbifold. We let $K$ denote the twisted I-bundle over the Klein bottle. Note that, 
while the Klein bottle is a  non-orientable surface, the space $K$ is an orientable $3$-manifold with a torus boundary. 

\begin{lemma}\label{all-SF-flat}
Let $M$ be an irreducible $3$-manifold, and assume that all the pieces in the JSJ decomposition are Seifert fibered
with Euclidean base orbifold. Then either:
\begin{enumerate}
\item $M$ is a torus bundle over $S^1$, or
\item $M$ consists of two copies of $K$ glued together along their boundary.
\end{enumerate}
\end{lemma}

\begin{proof}
From Proposition \ref{prop:bdry:flat:orb}, we know the only such Seifert fibered 
pieces are either (i) the torus times an interval, and (ii) the twisted $I$-bundle $K$ over the Klein bottle. 
If we have a piece of type (i) whose boundary tori are distinct in $M$, then we would violate the minimality
of the number of tori in the JSJ decomposition. So if we have a piece of type (i), then the JSJ decomposition
of $M$ in fact has a {\it single} piece, and $M$ must be a torus bundle over $S^1$. If there are no pieces
of type (i), then the decomposition of $M$ consists of two copies of $K$ identified together. 
\end{proof}

We will now compute the virtually cyclic geometric dimension for these classes of manifolds.

\subsection{Torus bundles over the circle.} 

\begin{proposition}\label{torus-bundle-over-circle}
Let $M$ be a torus bundle over $S^1$, with fundamental group $\Gamma$. Then exactly one of the following happens:
\begin{enumerate}
    \item $M$ is modeled on $\dbE^3$, hence $\gdvcgamma=4$, 
    \item $M$ is modeled on $\nil$, and $\gdvcgamma=3$, 
    \item $M$ is modeled on $\sol$, and $\gdvcgamma=3$.
\end{enumerate}
\end{proposition}

\begin{proof}
Since $M$ is a torus bundle over $S^1$, 
$\Gamma \cong \dbZ^2 \rtimes_{\varphi} \dbZ$ with $\varphi\colon \dbZ \to SL_2(\dbZ)$. Denote $\varphi(1)=A$. 
We have three cases depending on whether the matrix $A$ is elliptic, parabolic or hyperbolic (these cases correspond to
whether the trace of $A$ is $<2$, $=2$ and $>2$ respectively). 

If $A$ is elliptic it has finite order, so $\Gamma$ is virtually $\dbZ^3$. This implies $M$ 
is finitely covered by the $3$-torus, and must be crystallographic (see \cite[Table 1]{AFW15}). Then  
$\gdvcgamma=4$ follows from \cite{CFH06}.

If $A$ is parabolic, then the action on $\dbZ^2$ has an invariant rank one subgroup. This implies that the center $Z$ of $\Gamma'$ is infinite cyclic. 
Since $Z$ is a characteristic group of $\Gamma'$, it follows that $Z$ is an infinite cyclic normal subgroup of $\Gamma$. Applying 
\cite[Theorem~7]{Po08}, we see that $M$ is Seifert fibered with virtually nilpotent (but not virtually abelian) fundamental group, 
so is modeled on $\nil$ (see \cite[Table 1]{AFW15}). From Proposition \ref{empty:bdary:flat:orb} we obtain that $\gdvcgamma=3$.

Finally, consider the case where $A$ is hyperbolic. Then the action of $A$ on $\dbZ^2$ does not have any non-trivial invariant subgroups. 
This implies the center of $\Gamma$ is trivial, so by \cite[Theorem~4.7.13]{Th97} we obtain that $M$ is modeled on $\sol$. 
In order to compute $\gdvcgamma$, we will verify that every finite index subgroup of $\Gamma$ has finite center.
Let $H\leq \Gamma$ be a finite index subgroup. Then we have the short exact sequence
\[
1\to H\cap \dbZ^2 \to H \to \dbZ \to 1.
\]
Then $H\cap \dbZ^2$ has finite index in $\dbZ^2$, and $p(H)\neq 0$.
$A$ is hyperbolic, so every positive power of $A$ is also hyperbolic. This again implies that the centralizer $Z_{\Gamma}(H)$ must be trivial. 
Recalling that $H$ was an arbitrary finite index subgroup of $\Gamma$, \cite[Theorem~5.13]{LW12} allows us to conclude $\gdvcgamma=3$. 
\end{proof}

\subsection{Twisted doubles of $K$.}
\begin{proposition}\label{flat-flat:case}
Let $M$ be an irreducible $3$-manifold obtained as the union of two copies of $K$, where the gluing is 
via a homeomorphism $\varphi\colon T\to T$ between the boundary torus. Denote by $\Gamma$ 
the fundamental group of $M$. Then exactly one of the following happens:
\begin{enumerate}
    \item $M$ is modeled on $\dbE^3$, hence $\gdvcgamma=4$, 
    \item $M$ is modeled on $\nil$, and $\gdvcgamma=3$, 
    \item $M$ is modeled on $\sol$, and $\gdvcgamma=3$.
\end{enumerate}
\end{proposition}

\begin{proof}
We know that $\pi_1(K)$ is isomorphic to the fundamental group of the Klein bottle $\dbZ\rtimes \dbZ$.
This implies $\Gamma\cong(\dbZ\rtimes \dbZ) *_{\dbZ^2} (\dbZ\rtimes \dbZ)$, where the $\dbZ^2$, 
which embeds as an index two subgroup of $\dbZ\rtimes \dbZ$, comes from the boundary of $K$. Note that 
the $\dbZ ^2$ subgroup is a normal subgroup of $\Gamma$, so can be identified with the kernel of an
induced surjective morphism $p: \Gamma \rightarrow D_\infty$. Defining $\Gamma^\prime := p^{-1}(\mathbb Z)$ to
be the pre-image of the cyclic index two subgroup of the infinite dihedral group $D_\infty$, we obtain the diagram
\[
\xymatrix{
1 \ar[r]  & \dbZ^2 \ar@{=}[d] \ar[r] & \Gamma \ar[r]^{p} & D_\infty \ar[r] & 1 \\
1 \ar[r] & \dbZ^2 \ar[r] & \Gamma' \ar[r]^{p} \ar@{^{(}->}[u] & \dbZ \ar[r] \ar@{^{(}->}[u] & 1
}
\]
Thus we see that $\Gamma'$ is one of the groups discussed in Proposition \ref{torus-bundle-over-circle}. Since
$\Gamma$ contains $\Gamma'$ as an index two subgroup, we see that the geometry of $M$ coincides with 
the geometry of the corresponding double cover $M'$ (see the algebraic criteria in \cite[Table 1]{AFW15}). 

The calculations of $\gdvcgamma$ then follow from \cite{CFH06} in the $\dbE^3$ case, and from Proposition \ref{empty:bdary:flat:orb}
in the $\nil$ case. In the $\sol$ case, just as in Proposition \ref{torus-bundle-over-circle}, one can easily verify
that $\Gamma$ satisfies the conditions of \cite[Theorem~5.13]{LW12}, which gives us $\gdvcgamma=3$ (the
details are left to the reader). 
\end{proof}

Let us summarize the information we have so far on the JSJ decomposition of $M$, when $M$ is not geometric. 

\begin{corollary}\label{excluded-cases}
Let $M$ be a prime $3$-manifold, which we assume is not geometric, and let $N_i$ be the pieces in the JSJ decomposition of $M$.  
Then all the $N_i$ are either (i) hyperbolic, (ii) Seifert fibered over a hyperbolic base, or (iii) copies of $K$, the
twisted $I$-bundle over the Klein bottle. 

Moreover, every piece of type (iii) is attached to a piece of type (i) or (ii).
In particular, there must be a piece of type (i) or (ii).
\end{corollary}

\begin{proof}
There must be at least one torus in the decomposition, for otherwise $M$ itself is closed hyperbolic or closed Seifert fibered, 
hence geometric. By Lemma \ref{all-SF-flat} and Proposition \ref{torus-bundle-over-circle}, there are no pieces homeomorphic 
to $T^2\times [0,1]$, so the only pieces that
are Seifert fibered over a Euclidean base $2$-orbifold are copies of $K$. Finally, if a piece of type (iii) is attached to a piece
of type (iii), then Proposition \ref{flat-flat:case} tells us $M$ is geometric. 
\end{proof}


\section{Reducing to the JSJ pieces}\label{sec-reduction}


Next we relate the study of the virtually cyclic geometric dimension of the fundamental group of a {\it prime} manifold, to that of the components in its JSJ decomposition.
In Sections \ref{sec-Seifert-fibered} and \ref{sec-hyperbolic}, we have already calculated the virtually cyclic geometric dimension of the prime manifolds that are geometric. So throughout this section, we
will work exclusively with non-geometric prime $3$-manifolds.

\begin{theorem}\label{reduction:seifert:hyp}
Let $M$ be a closed, oriented, connected, \emph{prime} $3$-manifold which is not geometric. Let $N_1,\dots, N_k$ with $k\geq 1$, 
be the components arising in the JSJ decomposition. Denote $G=\pi_1(M)$, $G_i=\pi_1(N_i)$. Let $X_i$ be an arbitrary 
model for $\underline{\underline{E}}G_i$. Then
\[
3\leq \max\{\gdvc(G_i)|1\leq i\leq k\}\leq \gdvc(G) \leq \max\{4, \dim(X_i) | 1\leq i\leq k\}.
\]
\end{theorem}
\begin{proof}
Since $M$ is not geometric, it has at least one torus in its JSJ decomposition (see Corollary \ref{excluded-cases}), so we will 
have a subgroup of $G$ isomorphic to $\dbZ^2$. 
Moreover, every $G_i$ has a subgroup isomorphic to $\dbZ^2$, giving us the first inequality. The second inequality follows from Lemma 
\ref{properties:gd}. For the last inequality, we proceed as in the proof of Theorem \ref{reduction:irreducible}. The JSJ decomposition 
provides a splitting of $G$ as the fundamental group of a graph of groups with vertex groups $G_i$ and edge groups copies of $\dbZ^2$, 
and by Lemma \ref{properties:gd} we have $\gdvc(\dbZ^2)=3$. Now the conclusion will follow from Corollary \ref{vcgd:acyl} once we prove 
that the splitting of $G$ is acylindrical, which is done below in Proposition \ref{jsj acylindrical}.
\end{proof}

\begin{proposition}
\label{jsj acylindrical}
Let $M$ be a closed, oriented, connected, \emph{prime} 3-manifold, which is not geometric.
Let $\mathbf{Y}$ be the graph of groups associated to its JSJ decomposition. Then the splitting of $G=\pi_1(M)$
as the fundamental group of $\mathbf{Y}$ is acylindrical.
\end{proposition}
\begin{proof}
Let $T$ be the Bass-Serre covering tree of $\mathbf{Y}$, and let $c$ be a
path of length $5$.  $G$ acts without inversion on $T$, so
elements that stabilize $c$ must in fact fix it, for otherwise they would invert 
the center edge of $c$. We will argue that the stabilizer of $c$ is trivial. This 
will show that the splitting satisfies the definition of acylindricity, with integer $k=5$.

Let $c$ have edges $\{\widetilde{y}_1,\ldots,\widetilde{y}_5\}$, and
let $o(\widetilde{y}_i)=\widetilde{P}_{i-1}, t(\widetilde{y}_i)=\widetilde{P}_i$
in $T$. Let $y_i=p(\widetilde{y}_i)\in\edge Y$ and
$P_i=p(\widetilde{P}_i)\in\vrt Y$ for each vertex and edge in $c$. 
Let $N_i$, $0\leq i \leq 5$, be the manifolds that correspond to each 
vertex $P_i$, with fundamental groups $G_i$.  Let $T_i$, $1\leq i \leq 5$, 
be the torus associated to each edge $y_i$, and denote by $Z_i$ the
stabilizer of $\tilde{y_i}$ (which is a conjugate in $G$ of the 
fundamental group of $T_i$).

Now suppose that one of the $N_i$ ($1\leq i \leq 4$) is hyperbolic. Then the 
stabilizer of $c$ is contained in $Z_{i-1}\cap Z_i$. The groups $Z_{i-1}$ and
$Z_i$ are stabilizers of two distinct points in the boundary at infinity of $\mathbb
H^3$. It follows that the group that fixes $c$ must be finite, hence trivial since 
the $G_i$ is torsion-free.

\vskip 5pt

So we now need to consider the case where {\it all} the $N_i$ ($1\leq i \leq 4$) 
are Seifert fibered. Recall that the $N_i$ have non-empty boundary, so there 
are only two possible cases for each of their base orbifold: either the base is 
hyperbolic, or it is Euclidean. In view of Corollary \ref{excluded-cases}) either 
$N_2$ or $N_3$ is Seifert fibered with hyperbolic base $2$-orbifold.

\vskip 5pt

Let us now briefly pause and focus on $N$ a Seifert fibered space, with hyperbolic 
base $2$-orbifold. Then by \cite[Theorem~1.2.2]{Mo05}, $G=\pi_1(N)$ acts on 
$\hyp\times\mbr$, with quotient the corresponding $N_i$. Notice an important 
feature of such Seifert fibered spaces -- they come equipped with a {\it canonical} 
Seifert fibered structure. Indeed, the circle fibers in $N_i$ always lift to copies of 
the $\mbr$ factor in the universal cover.

Each edge incident to the corresponding vertex $\tilde P$ has stabilizer a $\mbz ^2$ 
subgroup of $G$. Up to reparametrization, the $\mbz ^2$-action on the universal 
cover $\hyp \times \mbr$ is described as follows. The first coordinate acts by 
translation in the $\mbr$-factor, while the second coordinate acts by a parabolic 
isometry on the $\hyp$-factor. Noting that a pair of parabolic isometries that are 
centered at different points at infinity always intersect trivially, we conclude that the 
corresponding pair of edge stabilizers can only intersect in an infinite cyclic subgroup. 
Moreover, the axes of translation of this cyclic subgroup corresponds precisely to the 
fibers of the Seifert fibration on $N$.

\vskip 5pt

We now continue our proof. If both $N_2$, $N_3$ are Seifert fibered with hyperbolic 
base, then we claim that $Z_1\cap Z_2 \cap Z_3$ is trivial. By the discussion above, 
$Z_1 \cap Z_2$ is an infinite cyclic subgroup of $G_2$, generated by the Seifert fibers 
of $N_2$. Similarly, $Z_2 \cap Z_3$ is also an infinite cyclic subgroup of $G_3$, 
generated by the Seifert fibers of $N_3$. Consider the torus $T_2$ (with fundamental 
group $Z_2$) where $N_2$ and $N_3$ are glued together. This $2$-torus has two 
circle fibrations induced on it, depending on whether we view it as a subspace of 
$N_2$ or of $N_3$. The circle fibers induced by the $N_2$ fibration correspond 
to the subgroup $Z_1\cap Z_2$, while the circle fibers induced by the $N_3$ 
fibration correspond to the subgroup $Z_2 \cap Z_3$. If these two subgroups 
intersect non-trivially, then the two fibrations match on the common $2$-torus 
$T_2$, and we obtain a Seifert fibered structure on $N_2 \cup _{T_2} N_3$. 
But this contradicts the minimality of the JSJ decomposition. Thus the 
two fibrations on $T_2$ {\it cannot} match, and hence 
$(Z_1\cap Z_2) \cap (Z_2 \cap Z_3)$ is trivial. But this group is precisely the 
intersection of the stabilizers of the three consecutive edges 
$\tilde y_1, \tilde y_2, \tilde y_3$ in the path $c$. Since this intersection contains 
the stabilizer of $c$, we conclude that $c$ has trivial stabilizer.

\vskip 10pt

Finally, we are left with one remaining case: one of $N_2$, $N_3$ is Seifert fibered 
with hyperbolic base, while the other one is Seifert fibered with flat base. Without 
loss of generality, we assume that $N_2$ has hyperbolic base, while $N_3$ has 
flat base. Then as discussed above, we see that $N_3$ must coincide with $K$, 
the twisted interval bundle over the Klein bottle. In particular, $G_3=\pi_1(N_3)$ 
coincides with the fundamental group of the Klein bottle, and the single boundary 
torus $\partial N_3$ has fundamental group $Z_2 = Z_3$ which is the canonical 
index two $\mbz ^2$ subgroup in $G_3$. 

Let us briefly focus on the manifold $K$. The universal cover of $K$ is $\mbr ^2 \times [0,1]$. 
There are precisely two possible Seifert fibrations of $K$. Indeed, a Seifert fibration lifts to a
foliation of the universal cover $\mbr ^2 \times [0,1]$ by parallel straight lines. In order to 
descend to a well-defined foliation of $K$, the straight lines have to be invariant under the 
action of the $\pi_1(K)$, the fundamental group of the Klein bottle. Since $\pi_1(K)$ is 
crystallographic, we have a well-defined holonomy involution $h: \dbZ^2 \rightarrow \dbZ^2$, 
given by conjugating the normal index two subgroup $\dbZ^2\triangleleft \pi_1(K)$ by any element in 
$\pi_1(K) \setminus \dbZ^2$. 
This holonomy action leaves invariant precisely two cyclic subgroups of $\dbZ^2$,
corresponding to the $\pm 1$ eigenspaces of $h$. The foliations with slopes matching
the eigenspace of $h$ are precisely the ones which will descend to $K$.

\vskip 5pt

Continuing our proof, the canonical Seifert fibered structure on the piece $N_2$ induces 
foliations by straight lines on $\mbr ^2 \times \{0\} \subset \tilde N_3$. It also induces a 
foliation by straight lines on $\mbr ^2 \times \{1\}\subset \tilde N_3$. These two foliations 
are related: the foliation on $\mbr ^2 \times \{1\}$ can be obtained 
from the foliation on $\mbr ^2 \times \{0\}$, by applying the holonomy map $h$. 
There are now two possible cases: either these foliations have slope matching
an eigenspace of $h$, or they will not.

If the foliation matches the eigenspace of $h$, then putting the corresponding Seifert
structure on $N_3$, we obtain a globally defined Seifert structure on $N_2 \cup_{T_2} N_3$.
This contradicts the minimality condition in the JSJ splitting.

On the other hand, if the slope does not match an eigenspace, then from the action 
of the holonomy $h$, we see that the straight line foliations on the two sides 
$\mbr ^2 \times \{0\}$ and $\mbr ^2 \times \{1\}$ are by lines of different slope. But this 
means that, if we view the infinite cyclic groups $Z_1\cap Z_2$ and $Z_3 \cap Z_4$ as subgroups 
of $\dbZ^2 \triangleleft \pi_1(K) = G_3$, they act by translations in distinct directions, and hence the intersection 
$(Z_1\cap Z_2) \cap (Z_3 \cap Z_4)$ is trivial. Since the stabilizer of $c$ is contained 
within this intersection, it is also trivial. This was the last remaining case, and hence 
completes the proof that the splitting is acylindrical.
\end{proof}


\section{Bredon cohomology computation}\label{sec:Bredon}

From our work in Sections \ref{sec-Seifert-fibered} and \ref{sec-hyperbolic}, we know the geometric dimension for fundamental groups of closed geometric
$3$-manifolds. In this section, we focus on non-geometric prime $3$-manifolds.

\begin{proposition}\label{prime-cases}
Let $M$ be a closed oriented prime $3$-manifold, with $\Gamma = \pi_1(M)$, and assume that $M$ is not geometric. Then $\gdvcgamma=3$.
\end{proposition}

\begin{proof}
Since $M$ is not geometric, Theorem \ref{reduction:seifert:hyp} gives us the inequality $3\leq \gdvcgamma \leq 4$, so we just need to rule out 
$\gdvcgamma = 4$.

Let $\mathbf{Y}$ be the graph of groups associated to the JSJ decomposition of $M$, so that $\pi_1(\mathbf{Y}) = \pi_1(M)=\Gamma$. 
Then the $G_i$ are the groups associated to the vertices $p_i\in\vrt Y$. Proposition \ref{jsj acylindrical} tells us that $\mathbf{Y}$ is an acylindrical 
graph of groups.  Letting $\fmly$ be the family of virtually cyclic subgroups of $\Gamma$ that are conjugate into one of the $G_i$, Proposition
\ref{push-out:acyl} tells us that an $\evc \Gamma$ can be obtained by attaching $2$-cells to a model for $E_{\calF}\Gamma$. Thus our proposition
would follow immediately from the 

\vskip 10pt

\noindent {\bf Claim:} There exists a $3$-dimensional model for $E_{\calF}\Gamma$.

\vskip 10pt

Unfortunately the na\"ive model for $E_{\calF}\Gamma$ described in the proof of Proposition \ref{bass serre construction} is $4$-dimensional.
In order to show that there exists a $3$-dimensional model, we will instead show that the fourth Bredon cohomology $H^4_{\fmly}(\Gamma;F)$ vanishes 
for all coefficient modules $F\in\textrm{Mod-}\mc{O}_{\fmly}\Gamma$. This implies that the Bredon cohomological dimension $\text{cd}_{\calF} \Gamma =3$,
which implies the existence of the desired $3$-dimensional model (see Lemma \ref{properties:gd}). 

To show $H^4_{\fmly}(\Gamma;F)=0$, we make use of the graph of spaces model described in Proposition \ref{bass serre construction}. 
Chose an orientation $A$ of the edges of the finite graph $\mathbf{Y}$. Then by \cite[Remark 4.2]{MP02}, the graph of spaces gives rise to the
long exact sequence
\begin{equation}\label{chiswell:les}
\ldots \stackrel{\partial^*}{\to} \bigoplus_{\vrt Y}\hvc^3(G_i;F)
\stackrel{\alpha^*}{\to} \bigoplus_{{y}\in A}\hvc^3(\mbz^2;F)
\stackrel{\iota^*}{\to} H^4_{\fmly}(\Gamma;F) \stackrel{\partial^*}{\to}
\bigoplus_{\vrt Y}\hvc^4(G_i;F)\to\ldots.
\end{equation}
We know that all the pieces in the JSJ decomposition satisfy $\gdvc(G_i)=3$, so by Lemma \ref{properties:gd}, we also have $\cdvc(G_i)=3$. 
This forces $\hvc^4(G_i;F)=0$ for all the $G_i$. Thus in order to prove that $H^4_{\fmly}(\Gamma;F)$ is trivial, it suffices to prove that  
\begin{equation}\label{surj}
\alpha^*:\bigoplus_{\vrt Y}\hvc^3(G_i;F)\to\bigoplus_{{y}\in A}\hvc^3(\mbz^2;F)
\end{equation}
is surjective. Given an oriented edge $y\in A$, and one of the endpoints $p_i \in \vrt Y$, the corresponding morphism $\hvc^3(G_i;F) \rightarrow \hvc^3(\mbz^2;F)$
is induced by the inclusion $\mbz^2 \hookrightarrow G_i$, corresponding to one of the boundary tori $T^2 \hookrightarrow N_i$. We know from Corollary
\ref{excluded-cases} that each piece $N_i$ from the JSJ decomposition has nonempty boundary and is either (i) hyperbolic, (ii) Seifert fibered with
hyperbolic base $2$-orbifold, or (iii) a copy of $K$, the twisted $I$-bundle over the Klein bottle. Let us analyze the morphism in the first two cases.
\vskip 10pt

In case (i), $N_i$ is hyperbolic with non-empty boundary and fundamental group $G_i$. Then Proposition \ref{hyperbolic:case} gives the following 
Mayer-Vietoris exact sequence
\[\ldots\to \hvc^3(G_i;F)\to
\bigg(\bigoplus_{H\in\mc{H}}\hvc^3(H;F)\bigg)\oplus
\hfin^3(G_i;F)\to\bigoplus_{H\in\mc{H}}\hfin^3(H;F)\to\ldots\]
Since each $H\in \mc{H}$ is either $2$-crystallographic or virtually cyclic, we always have that $\hfin^3(H;F)=0$. Also, among the elements of $\calH$ 
we have the fundamental groups of the boundary tori of $G_i$, lets call $\calH' \subset \calH$ the subset consisting of those copies of $\dbZ^2$. Then for every 
$\calH''\subseteq \calH'$ we obtain from the Mayer-Vietoris above that the map
\[
\hvc^3(G_i;F)\to
\bigoplus_{H\in\mc{H}''}\hvc^3(H;F)
\]
induced by the inclusion of subgroups is surjective. 

\vskip 10pt

\begin{table}
\begin{tabular}{| p{8cm} |c|p{1.5cm}|}
\hline 
Type of geometric piece & Analyzed in & $\gdvcgamma$\tabularnewline
\hline 
\hline 

Compact Seifert fibered piece with bad base orbifold or good base orbifold modeled on $S^2$ & Impossible, Lemma \ref{impossible:compact:case} &- \tabularnewline
\hline 

Compact Seifert fibered piece with base orbifold modeled on $\dbH^2$ & Proposition \ref{prop:bdry:hyp:orb} & $3$ \tabularnewline
\hline 

Compact Seifert fibered piece with base orbifold modeled on $\dbE^2$ & Proposition \ref{prop:bdry:flat:orb} & $3$ \tabularnewline
\hline 

Hyperbolic piece & Proposition \ref{hyperbolic:case}  & $3$ \tabularnewline
\hline 

\end{tabular}
\caption{Virtually cyclic dimension of pieces in the JSJ decomposition.}
\label{tabla2}
\end{table}

Next, let us analyze case (ii), where $N_i$ is Seifert fibered with $B$ modeled on $\dbH^2$ and fundamental group $G_i$. 
Then using the push-out from Proposition \ref{prop:bdry:hyp:orb}, an argument similar to the one in the hyperbolic case, shows 
that the map 
\[
\hvc^3(G_i;F)\to
\bigoplus_{H\in\mc{H}''}\hvc^3(H;F)
\]
is again surjective for every subset $\calH''\subset\calH$, where again $\calH$ is the set of $\dbZ^2$ subgroups in $G_i$ corresponding to boundary components
of $N_i$. 

Note that in case (iii), where $N_i$ is the twisted interval bundle over the Klein bottle, it is not clear how to prove 
a surjectivity statement as above, as we do not have a push-out construction for the corresponding classifying space. 

\vskip 10pt

We now return to the proof of the Proposition. We needed to show that the morphism $\alpha^*$ in Equation (\ref{surj}) is surjective. The only possible difficulty lies
from the $\dbZ^2$ subgroups that arise as boundaries of geometric pieces homeomorphic to $K$ (see last paragraph). But from Corollary \ref{excluded-cases}, 
every geometric piece $N_i$ that is homeomorphic to $K$ gets attached to another geometric piece 
$N_j$ that is {\bf not} homeomorphic to $K$. In particular, the corresponding morphism $\hvc^3(G_j;F)\to \hvc^3(H;F)$ is surjective (where $H\cong \dbZ^2$ is
the subgroup corresponding to the $2$-torus $\partial N_i$). It is now easy to see that the morphism in Equation (\ref{surj}) is in fact surjective, completing our proof.
\end{proof}


\section{Proof of the main theorem}\label{sec-main-thm}

We are now ready to establish our main theorem.

\begin{table}
\begin{tabular}{| p{8cm} |c|p{1.5cm}|p{1.5cm}|}
\hline 
Type of closed $3$-manifold & Analyzed in & $\gdvcgamma$ & Geometry\tabularnewline 
\hline 
\hline 
Seifert fibered with bad base orbifold or good base orbifold modeled on $S^2$  & Proposition \ref{bad:orb:case}  & 0 & $S^3$ or $S^2\times \dbE$ \tabularnewline 
\hline 
Seifert fibered manifold with base orbifold modeled on $\dbH^2$ & Proposition \ref{empty:bdary:hyp:orb} & 3 & $\dbH^2$ or $\pslt$ \tabularnewline
\hline 
 Seifert fibered manifold with base orbifold modeled on $\dbE^2$& Proposition \ref{empty:bdary:flat:orb} & $4$ or  $3$  & $\dbE^3$ or $\nil$ resp. \tabularnewline
\hline 

Hyperbolic manifold & Proposition \ref{hyperbolic:case}  & $3$ & $\dbH^3$ \tabularnewline
\hline 

\end{tabular}
\caption{Virtually cyclic dimension of closed geometric manifolds.}
\label{tabla}
\end{table}

\begin{proof}[Proof of Theorem \ref{main:theorem}]
First, we verify that $\gdvcgamma\leq 4$ for every closed orientable $3$-manifold $M$. In view of Theorem \ref{reduction:irreducible}, it is sufficient to consider
the case where $N$ is prime. We have two cases depending on whether $N$ is geometric or not. 
If $N$ is a closed geometric $3$-manifold, then $\gdvcgamma\leq 4$ always holds -- see Table \ref{tabla} for details. On the other hand, if $N$ is a prime 
$3$-manifold which is not geometric, then Proposition \ref{prime-cases} shows that $\gdvcgamma=3$. 

Having established that $\gdvcgamma\leq 4$ for all closed orientable $3$-manifolds, let us now analyze the possibilities for $\gdvcgamma$, and establish statements (1)-(4) in
our main theorem. 

\hskip 5pt

\noindent {\bf Statement (1).}  For every group $\Gamma$ and every family $\calF$ of subgroups we have that $\gd_\calF(\Gamma)=0$ if and only if $\Gamma\in \calF$. Statement (1) follows as a particular case.

\hskip 5pt

\noindent {\bf Statement (2).} Assume that $\gdvcgamma=2$. Then it follows from Theorem \ref{reduction:irreducible} that all the components in the prime 
decomposition have virtually cyclic geometric dimension at most $2$. Proposition \ref{prime-cases}
then tells us that all the prime factors of $M$ are geometric. Looking at Table \ref{tabla}, we see all the components in the 
prime decomposition must be modeled on 
$S^2\times \dbE$ or $S^3$, and hence have virtually cyclic fundamental group. Thus $\Gamma$ is a free product of virtually cyclic groups. 

Conversely, if $\Gamma$ is a free product of virtually cyclic groups, then we have an acylindrical splitting of $\Gamma$. By Corollary \ref{vcgd:acyl} we 
obtain $\gdvcgamma\leq 2$. To obtain a lower bound we just have to observe that $\Gamma$ always contains a free group on two generators. Since such 
groups have virtually cyclic dimension equal to $2$, we obtain $2 \leq \gdvcgamma$.

\hskip 5pt

\noindent {\bf Statement (3).} If $\Gamma$ contains a $\dbZ^3$ subgroup, applying
Lemma \ref{properties:gd} gives the lower bound $4 = \gdvc (\dbZ ^3) \leq\gdvcgamma$, which forces $\gdvcgamma=4$. Conversely, in view of Theorem 
\ref{reduction:irreducible}, if $\gdvcgamma=4$ then one of the components arising in the prime decomposition of $M$ must have virtually cyclic dimension $=4$.
But for prime manifolds, we know that having virtually cyclic dimension $=4$ implies that the manifold is geometric (by Proposition \ref{prime-cases}), and
looking at Table \ref{tabla} we see the manifold must be crystallographic. This implies its fundamental group (itself a subgroup of $\Gamma$) contains a 
$\dbZ ^3$ subgroup.

\hskip 5pt

\noindent {\bf Statement (4).} To complete the proof, let us now assume that $\Gamma$ is not virtually cyclic, nor a free product of virtually cyclic groups, 
nor has a $\dbZ ^3$ subgroup. We will prove that $\gdvcgamma=3$. Let $M=P_1 \# \cdots \#P_k$ with corresponding free splitting 
$\Gamma:= \Gamma_1 * \cdots * \Gamma_k$.  In view of Theorem \ref{reduction:irreducible}, it suffices to show that all the 
prime manifolds $P_i$ in the decomposition satisfy $\gdvc (\Gamma _i)\leq 3$, and that at least one 
$P_i$ has $\gdvc (\Gamma_i)=3$. 

First note that none of the $P_i$ can be crystallographic, since $\Gamma$ does not contain any $\dbZ ^3$ subgroup.
From Table \ref{tabla}, we see that if $P_i$ is prime, geometric, but not crystallographic, then $\gdvc (\Gamma _i)\leq 3$. 
On the other hand, if $P_i$ is not geometric, then from Proposition \ref{prime-cases} it must have $\gdvc (\Gamma _i)= 3$.
So we see that indeed all $\gdvc (\Gamma _i)\leq 3$.

Finally, if {\it none} of the $P_i$ have $\gdvc (\Gamma _i)= 3$, then they must all satisfy $\gdvc (\Gamma _i)\leq 2$. Combining the results in Table \ref{tabla}
and Proposition \ref{prime-cases}, this can only happen if {\it all} the $P_i$ satisfy $\gdvc (\Gamma _i)=0$. From Table \ref{tabla},
this only occurs if all the $\Gamma_i$ are virtually cyclic, forcing $\Gamma$ to either be virtually cyclic (if there is only one prime factor) or to be a free product
of virtually cyclic groups. But both of these statements are contradictions. We conclude that there must exist a $P_i$ with $\gdvc (\Gamma _i)= 3$. Applying
Theorem \ref{reduction:irreducible} gives us $\gdvcgamma=3$, and completes the proof of our main theorem.
\end{proof}

\begin{remark}
Looking through the proof of Theorem \ref{main:theorem} above, we see that
the exact same arguments also establish the geometric description given in Corollary \ref{main:corollary}.
\end{remark}

\bibliographystyle{alpha} 
\bibliography{myblib}
\end{document}